\documentclass[11pt]{article}
\usepackage{amsmath}
\usepackage{amssymb}
\usepackage{amsthm}
\usepackage{mathabx}
\usepackage[usenames]{color}
\usepackage{amscd}
\usepackage{dsfont}
\usepackage{indentfirst}

\usepackage[colorlinks=true,linkcolor=blue,filecolor=red,
citecolor=webgreen]{hyperref}
\definecolor{webgreen}{rgb}{0,.5,0}

\def\C{{\mathds{C}}}
\def\Q{{\mathds{Q}}}
\def\R{{\mathbb{R}}}
\def\N{{\mathds{N}}}
\def\Z{{\mathds{Z}}}
\def\Primes{{\mathds{P}}}
\def\1{{\bf 1}}

\def\pont{$\bullet$ }

\newtheorem{theorem}{Theorem}

\newtheorem{corollary}{Corollary}

\newtheorem{proposition}{Proposition}
\newtheorem{remark}{Remark}

\begin{document}

\title{{\bf Ramanujan expansions of arithmetic functions of several variables}}
\author{L\'aszl\'o T\'oth\\ \\ Department of Mathematics, University of P\'ecs \\
Ifj\'us\'ag \'utja 6, 7624 P\'ecs, Hungary \\ E-mail: {\tt ltoth@gamma.ttk.pte.hu}}
\date{}
\maketitle

\centerline{The Ramanujan Journal {\bf 47} (2018), Issue 3, 589--603}

\begin{abstract} We generalize certain recent results of Ushiroya concerning Ramanujan expansions of arithmetic
functions of two variables. We also show that some properties on expansions of arithmetic
functions of one and several variables using classical and unitary Ramanujan sums, respectively, run parallel.
\end{abstract}

{\sl 2010 Mathematics Subject Classification}: 11A25, 11N37

{\sl Key Words and Phrases}: Ramanujan expansion of arithmetic functions, arithmetic function of several variables, multiplicative function,
unitary divisor, unitary Ramanujan sum

\section{Introduction}
Basic notations, used throughout the paper are included in Section \ref{Sect_Not}. Further notations are explained by
their first appearance. Refining results of Wintner \cite{Win1944},
Delange \cite{Del1976} proved the following general theorem concerning Ramanujan expansions of arithmetic functions.

\begin{theorem} \label{Th_Delange} Let $f:\N \to \C$ be an arithmetic function. Assume that
\begin{equation} \label{cond_one_var}
\sum_{n=1}^{\infty} 2^{\omega(n)} \frac{|(\mu*f)(n)|}{n} < \infty.
\end{equation}

Then for every $n\in\N$ we have the absolutely convergent Ramanujan expansion
\begin{equation} \label{f_Raman_exp_one_var}
f(n) = \sum_{q=1}^{\infty} a_{q} c_{q}(n),
\end{equation}
where the coefficients $a_q$ are given by
\begin{equation*}
a_q = \sum_{m=1}^{\infty} \frac{(\mu*f)(mq)}{mq} \quad (q\in \N).
\end{equation*}
\end{theorem}

Delange also pointed out how this result can be formulated for
multiplicative functions $f$. By Wintner's theorem (\cite[Part I]{Win1944}, also see, e.g., Postnikov \cite[Ch.\, 3]{Pos1988},
Schwarz and Spilker \cite[Ch.\, II]{SchSpi1994}), condition \eqref{cond_one_var} ensures that the mean value
\begin{equation*}
M(f) = \lim_{x\to \infty} \frac1{x} \sum_{n\le x} f(n)
\end{equation*}
exists and $a_1=M(f)$.

Recently, Ushiroya \cite{Ush2016} proved the analog of Theorem \ref{Th_Delange} for arithmetic functions of two variables and as applications,
derived Ramanujan expansions of certain special functions. For example, he deduced that for any fixed $n_1,n_2\in \N$,
\begin{equation} \label{exp_sigma_two_var}
\frac{\sigma((n_1,n_2))}{(n_1,n_2)}= \zeta(3) \sum_{q_1,q_2=1}^{\infty} \frac{c_{q_1}(n_1)c_{q_2}(n_2)}{[q_1,q_2]^3},
\end{equation}
\begin{equation} \label{exp_tau_two_var}
\tau((n_1,n_2))= \zeta(2) \sum_{q_1,q_2=1}^{\infty} \frac{c_{q_1}(n_1)c_{q_2}(n_2)}{[q_1,q_2]^2}.
\end{equation}

Here \eqref{exp_sigma_two_var} corresponds to the classical identity
\begin{equation} \label{exp_sigma_one_var}
\frac{\sigma(n)}{n}= \zeta(2) \sum_{q=1}^{\infty} \frac{c_q(n)}{q^2} \quad (n\in \N),
\end{equation}
due to Ramanujan \cite{Ramanujan1918}. However, \eqref{exp_tau_two_var} has not a direct one dimensional analog. The identity of Ramanujan
for the divisor function, namely,
\begin{equation*}
\tau(n) =  - \sum_{q=1}^{\infty} \frac{\log q}{q} c_q(n) \quad (n\in \N)
\end{equation*}
cannot be obtained by the same approach, since the mean value $M(\tau)$ does not exist.

We remark that prior to Delange's paper, Cohen \cite{Coh1961} pointed out how absolutely convergent expansions \eqref{f_Raman_exp_one_var},
including \eqref{exp_sigma_one_var} can be deduced for some special classes of multiplicative functions of one variable.

It is easy to see that the same arguments of the papers by Delange \cite{Del1976} and Ushiroya \cite{Ush2016} lead to the extension of
Theorem \ref{Th_Delange} for Ramanujan expansions of (multiplicative) arithmetic functions of $k$ variables, for any $k\in \N$. But in order
to obtain $k$ dimensional versions of the identities \eqref{exp_sigma_two_var} and \eqref{exp_tau_two_var}, the method given in
\cite{Ush2016} to compute the coefficients is complicated.

Recall that $d$ is a unitary divisor (or block divisor) of $n$ if $d\mid n$ and $(d,n/d)=1$. Various problems concerning functions associated to
unitary divisors were discussed by several authors in the literature. Many properties of the functions $\sigma^*(n)$ and
$\tau^*(n)=2^{\omega(n)}$, representing the sum, respectively the number of unitary divisors of $n$ are similar to $\sigma(n)$ and $\tau(n)$.
Analogs of Ramanujan sums defined by unitary divisors, denoted by $c^*_q(n)$ and called unitary Ramanujan sums, are also known in the literature.
However, we are not aware of any paper concerning expansions of functions with respect to the sums $c^*_q(n)$. The only such paper we found is by
Subbarao \cite{Sub1966}, but it deals with a series over the other argument, namely $n$ in $c^*_q(n)$. See Section \ref{Sect_Raman_unit}.

In this paper we show that certain properties on expansions of functions of one and several variables using classical and unitary
Ramanujan sums, respectively, run parallel. We also present a simple approach to deduce Ramanujan expansions of both types for $g((n_1,\ldots,n_k))$, where $g:\N\to \C$ is a certain function. For example, we have the identities

\begin{equation*}
\frac{\sigma(n)}{n}= \zeta(2) \sum_{q=1}^{\infty} \frac{\phi_2(q)c^*_q(n)}{q^4} \quad (n\in \N),
\end{equation*}
\begin{equation} \label{sigma_two_var_unit}
\frac{\sigma((n_1,n_2))}{(n_1,n_2)}= \zeta(3) \sum_{q_1,q_2=1}^{\infty} \frac{\phi_3([q_1,q_2])c^*_{q_1}(n_1)c^*_{q_2}(n_2)}{[q_1,q_2]^6}
\quad (n_1,n_2\in \N),
\end{equation}
\begin{equation} \label{tau_two_var_unit}
\tau((n_1,n_2))= \zeta(2) \sum_{q_1,q_2=1}^{\infty} \frac{\phi_2([q_1,q_2])c^*_{q_1}(n_1)c^*_{q_2}(n_2)}{[q_1,q_2]^4} \quad (n_1,n_2\in \N),
\end{equation}
which can be compared to \eqref{exp_sigma_one_var}, \eqref{exp_sigma_two_var} and \eqref{exp_tau_two_var}, respectively.

We point out that there are identities concerning the sums $c^*_q(n)$, which do not have simple counterparts in the classical case. For example,
\begin{equation} \label{id_Piltz_unit_two}
\tau_3((n_1,n_2))
\end{equation}
\begin{equation*}
 = (\zeta(2))^2 \sum_{q_1,q_2=1}^{\infty}
\frac{\tau([q_1,q_2]) (\phi_2([q_1,q_2]))^2}{[q_1,q_2]^6} c^*_{q_1}(n_1)c^*_{q_2}(n_2) \quad (n_1,n_2\in \N).
\end{equation*}

For general accounts on classical Ramanujan sums and Ramanujan expansions of functions of one variable we refer to the book by
Schwarz and Spilker \cite{SchSpi1994} and to the survey papers by Lucht \cite{Luc2010} and  Ram~Murty \cite{Ram2013}.
See Sections \ref{Sect_Funct_unit} and \ref{Sect_Raman_unit} for properties and references on functions defined by unitary divisors and
unitary Ramanujan sums. Section \ref{Sect_func_several_var} includes the needed background on arithmetic functions of several variables.
Our main results and their proofs are presented in Section \ref{Sect_Results}.

\section{Notations} \label{Sect_Not}

\pont $\N=\{1,2,\ldots\}$, $\N_0=\{0,1,2,\ldots\}$,

\pont $\Primes$ is the set of (positive) primes,

\pont the prime power factorization of $n\in \N$ is $n=\prod_{p\in \Primes}
p^{\nu_p(n)}$, the product being over the primes $p$, where all but
a finite number of the exponents $\nu_p(n)$ are zero,

\pont $(n_1,\ldots,n_k)$ and $[n_1,\ldots,n_k]$ denote the greatest common divisor and the least common multiple, respectively
of $n_1,\ldots,n_k\in \N$,

\pont $(f*g)(n)=\sum_{d\mid n} f(d)g(n/d)$ is the Dirichlet convolution of the functions $f,g:\N \to \C$,

\pont $\omega(n)$ stands for the number of distinct prime divisors of $n$,

\pont $\mu$ is the M\"obius function,

\pont $\sigma_s(n)=\sum_{d\mid n} d^s$ ($s\in \R$),

\pont $\sigma(n)=\sigma_1(n)$ is the sum of divisors of $n$,

\pont $\tau(n)=\sigma_0(n)$ is the number of divisors of $n$,

\pont $\phi_s$ is the Jordan function of order $s$ given by
$\phi_s(n)=n^s\prod_{p\mid n} (1-1/p^s)$ ($s\in \R$),

\pont $\phi=\phi_1$ is Euler's totient function,

\pont $\psi_s$ is the generalized Dedekind function given by $\psi_s(n)=n^s \prod_{p\mid n} (1+1/p^s)$ ($s\in \R$),

\pont $\psi=\psi_1$ is the Dedekind function,

\pont $c_q(n)=\sum_{1\le k \le q, (k,q)=1} \exp(2\pi ikn/q)$
 are the Ramanujan sums ($q,n\in \N$),

\pont $\tau_k(n)$ is the Piltz divisor function, defined as the number of ways that $n$ can be written as a product of
$k$ positive integers,

\pont $\lambda(n)=(-1)^{\Omega(n)}$ is the Liouville function, where $\Omega(n)=\sum_p \nu_p(n)$,

\pont $\zeta$ is the Riemann zeta function,

\pont $d\mid\mid n$ means that $d$ is a unitary divisor of $n$, i.e., $d\mid n$ and $(d,n/d)=1$.
(We remark that this is in concordance with the standard notation $p^{\nu} \mid\mid n$ used for prime powers $p^\nu$.)

\section{Preliminaries}

\subsection{Functions of one variable defined by unitary divisors} \label{Sect_Funct_unit}

The study of arithmetic functions defined by unitary divisors goes back to Vaidyanathaswamy \cite{Vai1931} and Cohen \cite{Coh1960}.
The functions $\sigma^*(n)$ and $\tau^*(n)$ were already mentioned in the Introduction. The analog of Euler's $\phi$ function is $\phi^*$,
defined by $\phi^*(n)=\# \{k\in \N: 1\le k \le n, (k,n)_*=1\}$, where
\begin{equation*}
(k,n)_*=\max \{d: d\mid k, d\mid\mid n\}.
\end{equation*}

Note that $d\mid\mid (k,n)_*$  holds if and only if $d\mid k$ and $d\mid\mid n$. The functions $\sigma^*$, $\tau^*$,
$\phi^*$ are all multiplicative and $\sigma^*(p^\nu)= p^{\nu}+1$, $\tau^*(p^\nu)=2$, $\phi^*(p^\nu)=p^{\nu}-1$ for any prime powers $p^{\nu}$
($\nu \ge 1$).

The unitary convolution of the functions $f$ and $g$ is
\begin{equation*}
(f\times g)(n)=\sum_{d\mid\mid n} f(d)g(n/d) \quad (n\in \N),
\end{equation*}
it preserves the multiplicativity of functions, and the inverse of the constant $1$ function under the unitary convolution is $\mu^*$,
where $\mu^*(n)=(-1)^{\omega(n)}$, also multiplicative. The set ${\cal A}$ of arithmetic functions forms a unital commutative ring
with pointwise addition and the unitary convolution, having divisors of zero.

See, e.g., Derbal \cite{Der2006}, McCarthy \cite{McC1986}, Sitaramachandrarao and Suryanarayana
\cite{SitSur1973}, Snellman \cite{Sne2004} for properties and generalizations of functions associated to unitary divisors.

\subsection{Unitary Ramanujan sums} \label{Sect_Raman_unit}

The unitary Ramanujan sums $c^*_q(n)$ were defined by Cohen \cite{Coh1960} as follows:
\begin{equation*}
c^*_q(n) = \sum_{\substack{1\le k \le q \\ (k,q)_*=1}} \exp(2\pi ikn/q) \quad (q,n\in \N).
\end{equation*}

The identities
\begin{equation*}
c^*_q(n) = \sum_{d\mid\mid (n,q)_*} d\mu^*(q/d) \quad (q,n\in \N),
\end{equation*}
\begin{equation} \label{sum_unit_Raman_sum}
\sum_{d\mid\mid q} c^*_d(n) = \begin{cases} q, & \text{ if $q\mid n$,}\\ 0, & \text{ otherwise}
\end{cases}
\end{equation}
can be compared to the corresponding ones concerning the Ramanujan sums $c_q(n)$.

It turns out that $c^*_q(n)$ is multiplicative in $q$ for any fixed $n\in \N$, and
\begin{equation} \label{unit_Raman_prime_pow}
c^*_{p^\nu}(n) = \begin{cases} p^\nu-1, & \text{ if $p^\nu \mid n$,}\\ -1, & \text{ otherwise,}
\end{cases}
\end{equation}
for any prime powers $p^{\nu}$ ($\nu \ge 1$). Furthermore, $c^*_q(q)=\phi^*(q)$, $c^*_q(1)=\mu^*(q)$ ($q\in \N$).

If $q$ is squarefree, then the divisors of $q$ coincide with its unitary divisors. Therefore, $c^*_q(n)=c_q(n)$ for any squarefree
$q$ and any $n\in \N$. However, if $q$ is not squarefree, then there is no direct relation between these two sums.

\begin{proposition} \label{Prop_unit_Raman} For any $q,n\in \N$,
\begin{equation} \label{unit_Raman_abs_id}
\sum_{d\mid\mid q} |c^*_d(n)| = 2^{\omega(q/(n,q)_*)} (n,q)_*,
\end{equation}
\begin{equation} \label{unit_Raman_abs_ineq}
\sum_{d\mid\mid q} |c^*_d(n)| \le  2^{\omega(q)} n.
\end{equation}
\end{proposition}

\begin{proof} If $q=p^\nu$ ($\nu \ge 1$) is a prime power, then we have by \eqref{unit_Raman_prime_pow},
\begin{equation*}
\sum_{d\mid\mid p^\nu} |c^*_d(n)| = |c^*_1(n)| + |c^*_{p^\nu}(n)|= \begin{cases} 1+p^\nu-1=p^\nu, & \text{ if $p^\nu \mid n$,} \\ 1+1=2,
& \text{ otherwise.}
\end{cases}
\end{equation*}

Now \eqref{unit_Raman_abs_id} follows at once by the multiplicativity in $q$ of the involved functions, while \eqref{unit_Raman_abs_ineq} is its
immediate consequence.
\end{proof}

\begin{remark} \label{Class_Raman} {\rm The corresponding properties for the classical Ramanujan sums are
\begin{equation} \label{Raman_abs_id}
\sum_{d\mid q} |c_d(n)| = 2^{\omega(q/(n,q))} (n,q),
\end{equation}
\begin{equation} \label{Raman_abs_ineq}
\sum_{d\mid q} |c_d(n)| \le  2^{\omega(q)} n,
\end{equation}
where inequality \eqref{Raman_abs_ineq} is crucial in the proof of Theorem \ref{Th_Delange}, and identity \eqref{Raman_abs_id} was pointed out
by Grytczuk \cite{Gry1981}. }
\end{remark}

Of course, some other properties differ notably. For example, the unitary sums $c^*_q(n)$ do not enjoy the orthogonality property of the
classical Ramanujan sums, namely
\begin{equation*}
\frac1{n} \sum_{k=1}^n c_{q_1}(k)c_{q_2}(k) = \begin{cases} \phi(q), & \text{ if $q_1= q_2=q$,}\\ 0, & \text{ otherwise,}
\end{cases}
\end{equation*}
valid for any $q_1,q_2,n\in \N$ with $[q_1,q_2]\mid n$. To see this, let $p$ be a prime, let $q_1=p$, $q_2=p^2$ and $n=p^2$. Then,
according to \eqref{unit_Raman_prime_pow},
\begin{equation*}
\sum_{k=1}^{p^2} c^*_p(k)c^*_{p^2}(k) = p^2(p-1) \ne 0.
\end{equation*}

Let $\Lambda^*$ be the unitary analog of the von Mangoldt function $\Lambda$, given by
\begin{equation*}
\Lambda^*(k) = \begin{cases} \nu \log p, & \text{ if $k=p^\nu$,}\\ 0, & \text{ otherwise.}
\end{cases}
\end{equation*}

The paper of Subbarao \cite{Sub1966} includes the formula
\begin{equation} \label{series_Lambda_star}
\sum_{n=1}^{\infty} \frac{c^*_q(n)}{n} = - \Lambda^*(q) \quad (q\ge 2),
\end{equation}
with an incomplete proof, namely without showing that the series in \eqref{series_Lambda_star} converges. Here
\eqref{series_Lambda_star} is the analog of Ramanujan's formula (also see H\"older \cite{Hol1936})
\begin{equation*}
\sum_{n=1}^{\infty} \frac{c_q(n)}{n} = - \Lambda(q) \quad (q\ge 2).
\end{equation*}

For further properties and generalizations of unitary Ramanujan sums, see, e.g.,
Cohen \cite{Coh1960}, Johnson \cite{Joh1982}, McCarthy \cite{McC1986}, Suryanarayana \cite{Sur1970}. Note that a common
generalization of the sums $c_q(n)$ and $c^*_q(n)$, involving Narkiewicz type regular systems of divisors was
investigated by McCarthy \cite{McC1968} and the author \cite{Tot2004}.

\subsection{Arithmetic functions of several variables} \label{Sect_func_several_var}

For every fixed $k\in \N$ the set ${\cal A}_k$ of arithmetic functions $f:\N^k\to \C$ of
$k$ variables is an integral domain with pointwise addition and the Dirichlet convolution defined by
\begin{equation} \label{Dir_convo_sev_var}
(f*g)(n_1,\ldots,n_k)= \sum_{d_1\mid n_1, \ldots, d_k\mid n_k}
f(d_1,\ldots,d_k) g(n_1/d_1, \ldots, n_k/d_k),
\end{equation}
the unity being the function $\delta_k$, where
\begin{equation*}
\delta_k(n_1,\ldots,n_k)= \begin{cases} 1, & \text{ if $n_1=\cdots =n_k=1$,}\\ 0, & \text{ otherwise.}
\end{cases}
\end{equation*}

The inverse of the constant $1$ function under \eqref{Dir_convo_sev_var} is $\mu_k$, given by
\begin{equation*}
\mu_k(n_1,\ldots,n_k)=\mu(n_1)\cdots \mu(n_k) \quad (n_1,\ldots,n_k\in \N),
\end{equation*}
where $\mu$ is the (classical) M\"obius function.

A function $f\in {\cal A}_k$ is said to be multiplicative if it is
not identically zero and
\begin{equation*}
f(m_1n_1,\ldots,m_kn_k)= f(m_1,\ldots,m_k) f(n_1,\ldots,n_k)
\end{equation*}
holds for any $m_1,\ldots,m_k,n_1,\ldots,n_k\in \N$ such that
$(m_1\cdots m_k,n_1\cdots n_k)=1$.

If $f$ is multiplicative, then it is determined by the values
$f(p^{\nu_1},\ldots,p^{\nu_k})$, where $p$ is prime and
$\nu_1,\ldots,\nu_k\in \N_0$. More exactly, $f(1,\ldots,1)=1$ and
for any $n_1,\ldots,n_k\in \N$,
\begin{equation*}
f(n_1,\ldots,n_k)= \prod_{p\in \Primes} f(p^{\nu_p(n_1)}, \ldots,p^{\nu_p(n_k)}).
\end{equation*}

Similar to the one dimensional case, the Dirichlet convolution \eqref{Dir_convo_sev_var} preserves the multiplicativity of
functions. The unitary convolution for the $k$ variables case can be defined and investigated, as well, but we will not need it in
what follows. See our paper \cite{Tot2014}, which is a survey on (multiplicative) arithmetic functions of several variables.

\section{Main results} \label{Sect_Results}

We first prove the following general result:

\begin{theorem} \label{Th_f_gen} Let $f:\N^k \to \C$ be an arithmetic function {\rm ($k\in \N$)}.
Assume that
\begin{equation} \label{cond_several_var}
\sum_{n_1,\ldots,n_k=1}^{\infty} 2^{\omega(n_1)+\cdots + \omega(n_k)} \frac{|(\mu_k*f)(n_1,\ldots,n_k)|}{n_1\cdots n_k}< \infty.
\end{equation}

Then for every $n_1,\ldots,n_k\in \N$,
\begin{equation} \label{f_Raman_exp_several}
f(n_1,\ldots,n_k) = \sum_{q_1,\ldots,q_k=1}^{\infty} a_{q_1,\ldots,q_k} c_{q_1}(n_1)\cdots c_{q_k}(n_k),
\end{equation}
and
\begin{equation} \label{f_Raman_exp_unit_several}
f(n_1,\ldots,n_k) = \sum_{q_1,\ldots,q_k=1}^{\infty} a^*_{q_1,\ldots,q_k} c^*_{q_1}(n_1)\cdots c^*_{q_k}(n_k),
\end{equation}
where
\begin{equation*}
a_{q_1,\ldots,q_k} = \sum_{m_1,\ldots,m_k=1}^{\infty} \frac{(\mu_k*f)(m_1q_1,\ldots,m_kq_k)}{m_1q_1\cdots m_kq_k},
\end{equation*}
\begin{equation} \label{coeff_unit_case}
a^*_{q_1,\ldots,q_k} = \sum_{\substack{m_1,\ldots,m_k=1\\(m_1,q_1)=1,\ldots,(m_k,q_k)=1}}^{\infty}
\frac{(\mu_k*f)(m_1q_1,\ldots,m_kq_k)}{m_1q_1\cdots m_kq_k},
\end{equation}
the series \eqref{f_Raman_exp_several} and \eqref{f_Raman_exp_unit_several} being absolutely convergent.
\end{theorem}

We remark that according to the generalized Wintner theorem, under conditions of Theorem \ref{Th_f_gen}, the mean value
\begin{equation*}
M(f)= \lim_{x_1,\ldots, x_k  \to \infty} \frac1{x_1\cdots x_k} \sum_{n_1\le x_1,\ldots,n_k\le x_k} f(n_1,\ldots,n_k)
\end{equation*}
exists and $a_{1,\ldots,1}=a^*_{1,\ldots,1} = M(f)$. See Ushiroya \cite[Th.\ 1]{Ush2012} and the author \cite[Sect.\ 7.1]{Tot2014}.

\begin{proof} We consider the case of the unitary Ramanujan sums $c^*_q(n)$. For any $n_1,\ldots,n_k\in \N$ we have, by using
property \eqref{sum_unit_Raman_sum},
\begin{equation*}
f(n_1,\ldots,n_k) = \sum_{d_1\mid n_1, \ldots, d_k\mid n_k} (\mu_k*f)(d_1,\ldots,d_k)
\end{equation*}
\begin{equation*}
= \sum_{d_1,\ldots,d_k=1}^{\infty} \frac{(\mu_k*f)(d_1,\ldots,d_k)}{d_1\cdots d_k}
\sum_{q_1\mid\mid d_1} c^*_{q_1}(n_1) \cdots \sum_{q_k\mid\mid d_k} c^*_{q_k}(n_k)
\end{equation*}
\begin{equation*}
= \sum_{q_1,\ldots,q_k=1}^{\infty} c^*_{q_1}(n_1) \cdots c^*_{q_k}(n_k) \sum_{\substack{d_1,\ldots,d_k=1\\
q_1\mid\mid d_1,\ldots,q_k\mid\mid d_k}}^{\infty} \frac{(\mu_k*f)(d_1,\ldots,d_k)}{d_1\cdots d_k},
\end{equation*}
giving expansion \eqref{f_Raman_exp_unit_several} with the coefficients \eqref{coeff_unit_case}, by
denoting $d_1=m_1q_1,\ldots,d_k=m_kq_k$. The rearranging of the terms is justified by the absolute convergence
of the multiple series, shown hereinafter:
\begin{equation*}
\sum_{q_1,\ldots,q_k=1}^{\infty} |a^*_{q_1,\ldots,q_k}| |c^*_{q_1}(n_1)| \cdots |c^*_{q_k}(n_k)|
\end{equation*}
\begin{equation*}
\le \sum_{\substack{q_1,\ldots,q_k=1\\m_1,\ldots,m_k=1\\(m_1,q_1)=1,\ldots,(m_k,q_k)=1}}^{\infty}
\frac{|(\mu_k*f)(m_1q_1,\ldots,m_kq_k)|}{m_1q_1\cdots m_kq_k} |c^*_{q_1}(n_1)| \cdots |c^*_{q_k}(n_k)|
\end{equation*}
\begin{equation*}
= \sum_{t_1,\ldots,t_k=1}^{\infty} \frac{|(\mu_k*f)(t_1,\ldots,t_k)|}{t_1\cdots t_k}
\sum_{\substack{m_1q_1=t_1\\ (m_1,q_1)=1}} |c^*_{q_1}(n_1)| \cdots  \sum_{\substack{m_kq_k=t_k\\ (m_k,q_k)=1}} |c^*_{q_k}(n_k)|
\end{equation*}
\begin{equation*}
\le n_1\cdots n_k \sum_{t_1,\ldots,t_k=1}^{\infty} 2^{\omega(t_1)+\cdots +\omega(t_k)} \frac{|(\mu_k*f)(t_1,\ldots,t_k)|}{t_1\cdots t_k}
<\infty,
\end{equation*}
by using inequality \eqref{unit_Raman_abs_ineq} and condition \eqref{cond_several_var}.

For the Ramanujan sums $c_q(n)$ the proof is along the same lines, by using inequality \eqref{Raman_abs_ineq}. In the case $k=2$
the proof of \eqref{f_Raman_exp_several} was given by Ushiroya \cite{Ush2016}.
\end{proof}

For multiplicative functions $f$, condition \eqref{cond_several_var} is equivalent to
\begin{equation} \label{cond_several_var_multipl}
\sum_{n_1,\ldots,n_k=1}^{\infty}  \frac{|(\mu_k*f)(n_1,\ldots,n_k)|}{n_1\cdots n_k}< \infty
\end{equation}
and to
\begin{equation} \label{cond_several_var_multipl_prod}
\sum_{p\in \Primes} \sum_{\substack{\nu_1,\ldots,\nu_k= 0\\ \nu_1+\ldots +\nu_k\ge 1}}^{\infty}
\frac{|(\mu_k*f)(p^{\nu_1},\ldots,p^{\nu_k})|}{p^{\nu_1+\cdots+ \nu_k}} < \infty.
\end{equation}

We deduce the following result:

\begin{corollary} \label{Cor_f_mult} Let $f:\N^k \to \C$ be a multiplicative function {\rm ($k\in \N$)}.
Assume that condition \eqref{cond_several_var_multipl} or  \eqref{cond_several_var_multipl_prod} holds.
Then for every $n_1,\ldots,n_k\in \N$ one has the absolutely convergent expansions \eqref{f_Raman_exp_several},
\eqref{f_Raman_exp_unit_several}, and the coefficients can be written as
\begin{equation*}
a_{q_1,\ldots,q_k} = \prod_{p\in \Primes} \quad \sum_{\nu_1\ge \nu_p(q_1),\ldots, \nu_k\ge \nu_p(q_k)}
\frac{(\mu_k*f)(p^{\nu_1},\ldots,p^{\nu_k})}{p^{\nu_1+\cdots + \nu_k}},
\end{equation*}
\begin{equation*}
a^*_{q_1,\ldots,q_k} = \prod_{p\in \Primes} \quad \sideset{}{'}\sum_{\nu_1,\ldots,\nu_k}
\frac{(\mu_k*f)(p^{\nu_1},\ldots,p^{\nu_k})}{p^{\nu_1+\cdots + \nu_k}},
\end{equation*}
where $\sum^{'}$ means that for fixed $p$ and $j$, $\nu_j$ takes all values $\nu_j\ge 0$ if $\nu_p(q_j)=0$, and takes only the value
$\nu_j=0$ if $\nu_p(q_j)\ge 1$.
\end{corollary}

\begin{proof} This is a direct consequence of Theorem \ref{Th_f_gen} and the definition of multiplicative functions.
In the cases $k=1$ and $k=2$, with the sums $c_q(n)$ see \cite[Eq.\ (7)]{Del1976} and \cite[Th.\ 2.4]{Ush2016},
respectively.
\end{proof}

Next we consider the case $f(n_1,\ldots,n_k)= g((n_1,\ldots,n_k))$.

\begin{theorem} \label{Th_f_g} Let $g:\N \to \C$ be an arithmetic function and let $k\in \N$. Assume that
\begin{equation} \label{cond_g_multipl}
\sum_{n=1}^{\infty} 2^{k\, \omega(n)} \frac{|(\mu*g)(n)|}{n^k} < \infty.
\end{equation}

Then for every $n_1,\ldots,n_k\in \N$,
\begin{equation} \label{f_Raman_g_gcd}
g((n_1,\ldots,n_k)) = \sum_{q_1,\ldots,q_k=1}^{\infty} a_{q_1,\ldots,q_k} c_{q_1}(n_1) \cdots c_{q_k}(n_k),
\end{equation}
\begin{equation} \label{f_Raman_unit_g_gcd}
g((n_1,\ldots,n_k)) = \sum_{q_1,\ldots,q_k=1}^{\infty} a^*_{q_1,\ldots,q_k} c^*_{q_1}(n_1) \cdots c^*_{q_k}(n_k)
\end{equation}
are absolutely convergent, where
\begin{equation*}
a_{q_1,\ldots,q_k} = \frac1{Q^k} \sum_{m=1}^{\infty} \frac{(\mu*g)(mQ)}{m^k},
\end{equation*}
\begin{equation*}
a^*_{q_1,\ldots,q_k} = \frac1{Q^k} \sum_{\substack{m=1\\ (m,Q)=1}}^{\infty} \frac{(\mu*g)(mQ)}{m^k},
\end{equation*}
with the notation $Q=[q_1,\ldots,q_k]$.
\end{theorem}

\begin{proof} We apply Theorem \ref{Th_f_gen}. The identity
\begin{equation*}
g((n_1,\ldots,n_k))= \sum_{d\mid n_1,\ldots,d\mid n_k} (\mu*g)(d)
\end{equation*}
shows that now
\begin{equation*}
(\mu_k*f)(n_1,\ldots,n_k)= \begin{cases} (\mu*g)(n), & \text{ if $n_1=\cdots = n_k=n$,}\\ 0, & \text{ otherwise.}
\end{cases}
\end{equation*}

In the unitary case the coefficients are
\begin{equation*}
a^*_{q_1,\ldots,q_k} = \sum_{\substack{n=1\\ m_1q_1=\cdots = m_kq_k=n\\(m_1,q_1)=1,\ldots,(m_k,q_k)=1}}^{\infty}
\frac{(\mu_k*f)(m_1q_1,\ldots,m_kq_k)}{m_1q_1\cdots m_kq_k}
\end{equation*}
\begin{equation*}
=  \sum_{\substack{n=1\\ q_1\mid\mid n, \ldots, q_k\mid\mid n}}^{\infty}
\frac{(\mu*g)(n)}{n^k},
\end{equation*}
and take into account that $q_1\mid\mid n,\ldots,q_k\mid\mid n$ holds if and only if $[q_1,\ldots,q_k]=Q\mid\mid n$, that is,
$n=mQ$ with $(m,Q)=1$.

In the classical case, namely for the coefficients $a_{q_1,\ldots,q_k}$, the proof is analog.
\end{proof}

If the function $g$ is multiplicative, then condition \eqref{cond_g_multipl} is equivalent to
\begin{equation} \label{cond_g_multipl_var}
\sum_{n=1}^{\infty}  \frac{|(\mu*g)(n)|}{n^k}< \infty
\end{equation}
and to
\begin{equation} \label{cond_g_multipl_prod}
\sum_{p\in \Primes} \sum_{\nu=1}^{\infty}
\frac{|(\mu*g)(p^{\nu})|}{p^{k\, \nu}} < \infty,
\end{equation}
and Theorem \ref{Th_f_g} can be rewritten according to Corollary \ref{Cor_f_mult}.
We continue, instead, with the following result, having direct applications to special functions.

\begin{corollary} \label{Cor_f_gcompl_multip_only_multip} Let $g:\N \to \C$ be a multiplicative function and let $k\in \N$. Assume that
condition \eqref{cond_g_multipl_var} or \eqref{cond_g_multipl_prod} holds.

If $\mu*g$ is completely multiplicative, then \eqref{f_Raman_g_gcd}
holds for every $n_1,\ldots,n_k\in \N$ with
\begin{equation*}
a_{q_1,\ldots,q_k} = \frac{(\mu*g)(Q)}{Q^k} \sum_{m=1}^{\infty} \frac{(\mu*g)(m)}{m^k}.
\end{equation*}

Identity \eqref{f_Raman_unit_g_gcd} holds for every $n_1,\ldots,n_k\in \N$ {\rm (and any multiplicative $g$)} with
\begin{equation*}
a^*_{q_1,\ldots,q_k} = \frac{(\mu*g)(Q)}{Q^k} \sum_{\substack{m=1\\(m,Q)=1}}^{\infty} \frac{(\mu*g)(m)}{m^k}.
\end{equation*}
\end{corollary}

\begin{proof} This is immediate by Theorem \ref{Th_f_g}. By making use of condition $(m,Q)=1$, it follows for the coefficients
$a^*_{q_1,\ldots,q_k}$ that $(\mu *g)(mQ)=(\mu *g)(m)(\mu *g)(Q)$ for any multiplicative function $g$.
\end{proof}

\begin{corollary} \label{Cor_sigma_s} For every $n_1,\ldots,n_k\in \N$ the following series are absolutely convergent:
\begin{equation*}
\frac{\sigma_s((n_1,\ldots,n_k))}{(n_1,\ldots,n_k)^s} = \zeta(s+k) \sum_{q_1,\ldots,q_k=1}^{\infty}
\frac{c_{q_1}(n_1)\cdots c_{q_k}(n_k)}{Q^{s+k}} \quad (s\in \R, s+k>1),
\end{equation*}
\begin{equation} \label{sigma_k}
\frac{\sigma((n_1,\ldots,n_k))}{(n_1,\ldots,n_k)} = \zeta(k+1) \sum_{q_1,\ldots,q_k=1}^{\infty}
\frac{c_{q_1}(n_1)\cdots c_{q_k}(n_k)}{Q^{k+1}} \quad (k\ge 1),
\end{equation}
\begin{equation} \label{tau_k}
\tau((n_1,\ldots,n_k)) = \zeta(k) \sum_{q_1,\ldots,q_k=1}^{\infty}
\frac{c_{q_1}(n_1)\cdots c_{q_k}(n_k)}{Q^k} \quad (k\ge 2).
\end{equation}
\end{corollary}

\begin{proof} Apply Corollary \ref{Cor_f_gcompl_multip_only_multip} for $g(n)=\sigma_s(n)/n^s$, where the function $(\mu*g)(n)=1/n^s$
is completely multiplicative.
\end{proof}

In the case $k=2$ the identities of Corollary \ref{Cor_sigma_s} were given in \cite[Ex.\ 3.8, 3.9]{Ush2016}. For $k=2$,
\eqref{sigma_k} and \eqref{tau_k} reduce to \eqref{exp_sigma_two_var} and \eqref{exp_tau_two_var}, respectively.

\begin{corollary} \label{Cor_phi_s_unit} For every $n_1,\ldots,n_k\in \N$ the following series are absolutely convergent:
\begin{equation*}
\frac{\sigma_s((n_1,\ldots,n_k))}{(n_1,\ldots,n_k)^s} = \zeta(s+k) \sum_{q_1,\ldots,q_k=1}^{\infty}
\frac{\phi_{s+k}(Q) c^*_{q_1}(n_1)\cdots c^*_{q_k}(n_k)}{Q^{2(s+k)}}
\end{equation*}
with $s\in \R, s+k>1$,
\begin{equation} \label{sigma_k_unit}
\frac{\sigma((n_1,\ldots,n_k))}{(n_1,\ldots,n_k)} = \zeta(k+1) \sum_{q_1,\ldots,q_k=1}^{\infty}
\frac{\phi_{k+1}(Q) c^*_{q_1}(n_1)\cdots c^*_{q_k}(n_k)}{Q^{2(k+1)}} \quad (k\ge 1),
\end{equation}
\begin{equation} \label{tau_k_unit}
\tau((n_1,\ldots,n_k)) = \zeta(k) \sum_{q_1,\ldots,q_k=1}^{\infty}
\frac{\phi_k(Q) c^*_{q_1}(n_1)\cdots c^*_{q_k}(n_k)}{Q^{2k}} \quad (k\ge 2).
\end{equation}
\end{corollary}

\begin{proof} Apply the second part of Corollary \ref{Cor_f_gcompl_multip_only_multip} for $g(n)=\sigma_s(n)/n^s$. Since $(\mu*g)(n)=1/n^s$, we
deduce that
\begin{equation*}
a^*_{q_1,\ldots,q_k} = \frac1{Q^{s+k}} \sum_{\substack{m=1\\(m,Q)=1}}^{\infty} \frac1{m^{s+k}}= \zeta(s+k) \frac{\phi_{s+k}(Q)}{Q^{2(s+k)}}.
\end{equation*}
\end{proof}

If $k=2$, then identities \eqref{sigma_k_unit} and \eqref{tau_k_unit} recover \eqref{sigma_two_var_unit}
and \eqref{tau_two_var_unit}, respectively.

Corollary \ref{Cor_f_gcompl_multip_only_multip} can be applied for several other special functions $g$. Further examples are
$g(n)=\beta_s(n)/n^s$ ($s\in \R$) and $g(n)=r(n)$, where $\beta_s(n)=\sum_{d\mid n} d^s\lambda(n/d)$ and $r(n)$ denotes the number of
solutions $(x,y)\in \Z^2$ of the equation $x^2+y^2=n$.

\begin{corollary} \label{Cor_beta_s} For every $n_1,\ldots,n_k\in \N$ and $s\in \R, s+k>1$ the following series are absolutely
convergent:
\begin{equation} \label{beta_1_1}
\frac{\beta_s((n_1,\ldots,n_k))}{(n_1,\ldots,n_k)^s} = \frac{\zeta(2(s+k))}{\zeta(s+k)} \sum_{q_1,\ldots,q_k=1}^{\infty}
\frac{\lambda(Q)c_{q_1}(n_1)\cdots c_{q_k}(n_k)}{Q^{s+k}},
\end{equation}
\begin{equation*}
\frac{\beta_s((n_1,\ldots,n_k))}{(n_1,\ldots,n_k)^s} = \frac{\zeta(2(s+k))}{\zeta(s+k)} \sum_{q_1,\ldots,q_k=1}^{\infty}
\frac{\lambda(Q)\psi_{s+k}(Q) c^*_{q_1}(n_1)\cdots c^*_{q_k}(n_k)}{Q^{2(s+k)}}.
\end{equation*}
\end{corollary}

In the case $k=s=1$, identity \eqref{beta_1_1} was derived by Cohen \cite[Eq.\ (12)]{Coh1961}. See also the author
\cite[Eq.\ (16)]{Tot2013}.

\begin{corollary} For every $n_1,\ldots,n_k\in \N$ \textup{($k\ge 1$)} the following series are absolutely convergent:
\begin{equation} \label{id_r_k}
r((n_1,\ldots,n_k)) = 4 L(\chi,k) \sum_{\substack{q_1,\ldots,q_k=1\\Q \text{ odd }}}^{\infty} \frac{(-1)^{(Q-1)/2} c_{q_1}(n_1)
\cdots c_{q_k}(n_k)}{Q^k},
\end{equation}
\begin{equation*}
r((n_1,\ldots,n_k)) = 4 L(\chi,k) \sum_{\substack{q_1,\ldots,q_k=1\\Q \text{ odd }}}^{\infty} \frac{(-1)^{(Q-1)/2}
F(Q)}{Q^k} c^*_{q_1}(n_1) \cdots c^*_{q_k}(n_k),
\end{equation*}
where
\begin{equation*}
L(\chi,k)= \sum_{n=0}^{\infty} \frac{(-1)^n}{(2n+1)^k}= \prod_{\substack{p\in \Primes \\ p>2}} \left(1-\frac{(-1)^{(p-1)/2}}{p^k} \right)^{-1},
\end{equation*}
$\chi=\chi_4$ being the nonprincipal character (mod $4$) and
\begin{equation*}
F(Q)= \prod_{\substack{p\mid Q\\ p>2}} \left(1-(-1)^{(p-1)/2}/p^k\right).
\end{equation*}
\end{corollary}

\begin{proof} Use that $r(n)=4\sum_{d\mid n} \chi(d)$, thus $\mu * r/4=\chi$, completely multiplicative.
Here \eqref{id_r_k} is well known for $k=1$, and was obtained in \cite[Ex. 3.13]{Ush2016} in the case $k=2$.
\end{proof}

If $g$ is multiplicative such that $\mu*g$ is not completely multiplicative, then Corollary \ref{Cor_f_gcompl_multip_only_multip}
can be applied for the sums $c^*_q(n)$, but not for $c_q(n)$, in general.

\begin{corollary} For every $n_1,\ldots,n_k\in \N$ the following series are absolutely convergent:
\begin{equation} \label{phi_s_several_var_unit}
\frac{\phi_s((n_1,\ldots,n_k))}{(n_1,\ldots,n_k)^s} = \frac1{\zeta(s+k)} \sum_{q_1,\ldots,q_k=1}^{\infty}
\frac{\mu(Q) c^*_{q_1}(n_1)\cdots c^*_{q_k}(n_k)}{\phi_{s+k}(Q)}
\end{equation}
with $s\in \R, s+k>1$,
\begin{equation} \label{phi_several_var_unit}
\frac{\phi((n_1,\ldots,n_k))}{(n_1,\ldots,n_k)} = \frac1{\zeta(k+1)} \sum_{q_1,\ldots,q_k=1}^{\infty}
\frac{\mu(Q) c^*_{q_1}(n_1)\cdots c^*_{q_k}(n_k)}{\phi_{k+1}(Q)} \quad (k\ge 1).
\end{equation}
\end{corollary}

\begin{proof} Similar to the proof of Corollary \ref{Cor_phi_s_unit}, selecting $g(n)=\phi_s(n)/n^s$, where $(\mu*g)(n)=\mu(n)/n^s$.
\end{proof}

However, observe that if $Q$ is not squarefree, then $\mu(Q)=0$ and all terms of the sums in
\eqref{phi_s_several_var_unit} and \eqref{phi_several_var_unit} are zero. Now if $Q$ is squarefree, then
$q_j$ is squarefree and $c^*_{q_j}(n_j)=c_{q_j}(n_j)$ for any $j$ and any $n_j$. Thus we deduce the next identities:

\begin{corollary} For every $n_1,\ldots,n_k\in \N$ the following series are absolutely convergent:
\begin{equation} \label{phi_s_several_var}
\frac{\phi_s((n_1,\ldots,n_k))}{(n_1,\ldots,n_k)^s} = \frac1{\zeta(s+k)} \sum_{q_1,\ldots,q_k=1}^{\infty}
\frac{\mu(Q) c_{q_1}(n_1)\cdots c_{q_k}(n_k)}{\phi_{s+k}(Q)}
\end{equation}
with $s\in \R, s+k>1$,
\begin{equation} \label{phi_several_var}
\frac{\phi((n_1,\ldots,n_k))}{(n_1,\ldots,n_k)} = \frac1{\zeta(k+1)} \sum_{q_1,\ldots,q_k=1}^{\infty}
\frac{\mu(Q) c_{q_1}(n_1)\cdots c_{q_k}(n_k)}{\phi_{k+1}(Q)} \quad (k\ge 1),
\end{equation}
which are formally the same identities as \eqref{phi_s_several_var_unit} and \eqref{phi_several_var_unit}.
\end{corollary}

Note that the special identities \eqref{phi_s_several_var} and \eqref{phi_several_var} can also be derived by the first part
of Corollary \ref{Cor_f_gcompl_multip_only_multip}, without considering the unitary sums
$c^*_q(n)$. See Ushiroya \cite[Ex.\ 3.10, 3.11]{Ush2016} for the case $k=2$, using different arguments.

Similar formulas can be deduced for the generalized Dedekind function $\psi_s(n)=n^s \prod_{p\mid n} (1+1/p^s)$ ($s\in \R$).
See Cohen \cite[Eq.\ (13)]{Coh1961} in the case $k=s=1$.

Finally, consider the Piltz divisor function $\tau_m$. The following formula, concerning the sums $c^*_q(n)$, deduced along the same lines with
the previous ones, has no simple counterpart in the classical case.

\begin{corollary} For any $n_1,\ldots,n_k,m\in \N$ with $m\ge 2, k\ge 2$,
\begin{equation} \label{id_Piltz_unit}
\tau_m((n_1,\ldots,n_k)) = (\zeta(k))^{m-1} \sum_{q_1,\ldots,q_k=1}^{\infty}
\frac{\tau_{m-1}(Q) (\phi_k(Q))^{m-1}}{Q^{km}} c^*_{q_1}(n_1)\cdots c^*_{q_k}(n_k).
\end{equation}
\end{corollary}

If $m=2$, then \eqref{id_Piltz_unit} reduces to formula \eqref{tau_k_unit}, while for $m=3$ and $k=2$ it gives
\eqref{id_Piltz_unit_two}.

\section{Acknowledgement}
The author thanks the anonymous referee for careful reading of the manuscript and helpful comments.

\end{document}